\documentclass[reqno, 10pt]{amsart}

\usepackage{amsfonts}
\usepackage{epsfig}
\usepackage{psfrag}
\usepackage[latin1]{inputenc} 

\newtheorem{theorem}{Theorem}[section]                                          
\newtheorem{proposition}[theorem]{Proposition}                          
\newtheorem{lemma}[theorem]{Lemma}

\newtheorem{remark}{Remark}[section]

\newcommand{\nn}{\nonumber}
\newcommand{\oo}{\infty}

\newcommand{\ra}{\rightarrow}

\newcommand{\Z}{\mathbb{Z}}

\makeatletter
\@addtoreset{equation}{section}
\makeatother

\renewcommand\thefigure{\thesection.\@arabic\c@figure}
\renewcommand\thetable{\thesection.\@arabic\c@table}

\bibdata{prob}
\bibstyle{alpha} 

\title[Generalized drainage network]{Convergence to the Brownian Web for a generalization of the drainage network model}

\author{Cristian Coletti and Glauco Valle}

\date{}

\address{
\newline
UFRJ - Departamento de m\'etodos estat\'{\i}sticos do Instituto de Matem\'atica.
\newline  Caixa Postal 68530, 21945-970, Rio de Janeiro, Brasil
\newline
e-mail:  \rm \texttt{glauco.valle@im.ufrj.br} 
\newline
\newline
UFABC - Centro de Matem\'atica, Computa\c{c}\~ao e Cogni\c{c}\~ao.
\newline
Avenida dos Estados, 5001, Santo Andr\'e - S\~ao Paulo, Brasil
\newline
e-mail:  \rm \texttt{cristian.coletti@ufabc.edu.br}
}

\thanks{Work supported by CNPq, FAPERJ and FAPESP grant 2009/52379-8}

\begin{document}
  
\maketitle

\begin{abstract}
We introduce a system of one-dimensional coalescing nonsimple random walks with long range jumps allowing crossing paths and exibiting dependence before coalescence. We show that under diffusive scaling this system converges in distribution to the Brownian Web.
\end{abstract}

\section{Introduction} 
\label{sec:intro}

The paper is devoted to the analysis of convergence in distribution for a diffusively rescaled system of one-dimensional coalescing random walks starting at each point in the space and time lattice $\mathbb{Z} \times \mathbb{Z}$. To not get repetitive, when we make mention to the convergence of a system of random walks, we always consider one-dimensional random walks and diffusive space time scaling. In this introduction, we aim at describing informally the system and explain why its study is relevant. 

First, the limit is the Brownian Web (BW) which is a system of coalescing Brownian motions starting at each point in the space and time plane $\mathbb{R} \times \mathbb{R}$. It is the natural scaling limit of a system of simple coalescing random walks. Here, the BW is a proper random element of a metric space whose points are compact sets of paths. This characterization of the BW was given by Fontes, Isopi, Newman, Ravishankar \cite{finr}, although formal descriptions of systems of coalescing Brownian motions had previously been considered, initially by Arratia \cite{a,a1} and then by T\'oth, Werner \cite{tw}. Following \cite{finr}, the study of convergence in distribution of systems of random walks to the BW and its variations has become an active field of research, for instance, see \cite{cfd,ffw,finr1,finr2,nrs,ss,ss1}.

In Fontes, Isopi, Newman, Ravishankar \cite{finr1} a general convergence criteria for the BW is obtained but only verified for a system of simple coalescing random walks that evolve independently up to the time of coalescence (no crossing paths). These criteria are used in Newman, Ravishankar, Sun \cite{nrs} to prove convergence to the Brownian Web of coalescing nonsimple random walks that evolve independently up to the time of coalescence (crossing paths). 

Coletti, Fontes, Dias \cite{cfd} have proved the convergenge of the BW for a system of coalescing Random Walks introduced by Gangopadhyay, Roy, Sarkar \cite{grs}. This system is called the drainage network. It evolves in the space time lattice $\mathbb{Z} \times \mathbb{Z}$ in the following way: Each site of $\mathbb{Z} \times \mathbb{Z}$ is considered open or closed according to an iid family of Bernoulli Random variables; If at time $z$ a walk is at site $x$ then at time $z+1$ its position is the open site nearest to $x$ (if we have two possible choices for the nearest open site, we choose one of them with probability $1/2$). The Drainage network is a system of coalescing nonsimple random walks with long range jumps where no crossing paths may occur. Futhermore the random walks paths are dependent before coalescence. These properties make the convergence of the drainage network to the BW a relevant and nontrivial example.

We propose here a natural generalization of the drainage network and prove the convergence to the BW. If at time $z$ a walk occupies the site $x$, then it chooses at random $k \in \mathbb{N}$ and jumps to the $k$th open site nearest to $x$, see the next section for the formal description. Now we have a system of coalescing nonsimple random walks with long range jumps and dependence before coalescence where crossing paths may occur. We suggest to the reader the discussion on \cite{nrs} about the addicional difficulties that arise in the study of convergence of systems of coalescing random walks when crossing paths may occur.

\medskip

We finish this section given a brief description of the BW which follows closely the description given in \cite{ss}, see also \cite{finr1} and the appendix in \cite{ss1}. Consider the extended plane $\bar{\mathbb{R}}^2 = [-\infty,\infty]^2$ as the completion of $\mathbb{R}^2$ under the metric
$$
\rho((x_1,t_1),(x_2,t_2)) = |\tanh(t_1) - \tanh(t_2) | \vee \Big| \frac{\tanh(x_1)}{1+|t_1|} - \frac{\tanh(x_2)}{1+|t_2|} \Big| \, ,
$$ 
and let $\bar{\rho}$ be the induced metric on $\bar{\mathbb{R}}^2$. In $(\bar{\mathbb{R}}^2,\bar{\rho})$, the lines $[-\infty,\infty] \times \{ \infty \}$ and $[-\infty,\infty] \times \{ - \infty \}$ correspond respectively to single points $(\star, \infty)$ and $(\star,-\infty)$, see picture 2 in \cite{ss1}. Denote by $\Pi$ the set of all continuous paths in $(\bar{\mathbb{R}}^2,\bar{\rho})$ of the form $\pi: t \in [\sigma_\pi,\infty] \rightarrow (f_\pi(t),t) \in (\bar{\mathbb{R}}^2,\bar{\rho})$ for some $\sigma_\pi \in [-\infty,\infty]$ and $f_\pi : [\sigma,\infty] \rightarrow [-\infty,\infty] \cup \{\star \}$. For $\pi_1$, $\pi_2 \in \Pi$, define $d(\pi_1,\pi_2)$ by  
$$
|\tanh(\sigma_{\pi_1}) - \tanh(\sigma_{\pi_2}) | \vee \sup_{t \ge \sigma_{\pi_1} \wedge \sigma_{\pi_2}} 
\Big| \frac{\tanh(f_{\pi_1}(t \vee \sigma_{\pi_1}))}{1+|t|} - \frac{\tanh(f_{\pi_2}(t \vee \sigma_{\pi_2}))}{1+|t|} \Big| \, ,
$$  
we have a metric in $\Pi$ such that $(\Pi,d)$ is a complete separable metric space. Now define $\mathcal{H}$ as the space of compact sets of $(\Pi,d)$ with the topology induced by the Hausdorff metric. Then $\mathcal{H}$ is a complete separable metric space. The Brownian web is a random element $\mathcal{W}$ of $\mathcal{H}$ whose distribution is uniquely characterized by the following three properties (see Theorem 2.1 in \cite{finr1}):
\begin{enumerate}
\item[(a)] For any deterministic $z \in \mathbb{R}^2$, almost surely there is a unique path $\pi_z$ of $\mathcal{W}$ that starts at $z$.
\item[(b)] For any finite deterministic set of points $z_1$, ... ,$z_k$ in $\mathbb{R}^2$, the collection $(\pi_1,...,\pi_n)$ is distributed as coalescing Brownian motions independent up to the time of coalescence.
\item[(c)] For any deterministic countable dense subset $\mathcal{D} \subset \mathbb{R}^2$, almost surely, $\mathcal{W}$ is the closure of $\{ \pi_z : z \in \mathcal{D} \}$ in $(\Pi,d)$.
\end{enumerate}

\medskip

This paper is organized as follows: In Section \ref{sec:model} we describe the generalized drainage network model, obtain some basic properties and state the main result of this text which says that it converges to the BW. The convergence is proved in Section \ref{sec:convergence} where Fontes, Isopi, Newmann, Ravishankar \cite{finr1} criteria is presented and discussed. As we shall explain, a carefull analysis of the proof presented in \cite{nrs} reduce our work to prove an estimate on the distribution of the coalescing time between two walks in the model, which is carried in subsection \ref{subsec:coaltimes}, and to prove that a diffusively rescaled system of $n$ walks in the model converge in distribution to $n$ coalescing Brownian motions which are independent up to the coalescence time, this is done in subsection \ref{subsec:I}. 

\medskip

\section{The generalized drainage network model} 
\label{sec:model}

Let $\mathbb{N}$ the set of positive integers. In order to present our model we need to introduce some notation:
\begin{itemize}
\item Let $(\omega(z))_{z \in \mathbb{Z}^2}$ be a family of independent Bernoulli random variables with parameter $p \in (0,1)$ and denote by $\textrm{P}_p$ the induced probability in $\{0,1\}^{\mathbb{Z}^2}$.
\item Let $(\theta (z))_{z \in \mathbb{Z}^2}$ be a family of independent Bernoulli random variables with parameter $1/2$ and denote by $\textrm{P}_{\frac{1}{2}}$ the induced probability in $\{0,1\}^{\mathbb{Z}^2}$.
\item Let $(\zeta (z))_{z \in \mathbb{Z}^2}$ be a family of independent and identically distributed random variables on $\mathbb{N}$ with probability function $q: \mathbb{N} \rightarrow [0,1]$ and denote by $\textrm{P}_{q(\cdot)}$ the induced probability in $\mathbb{N}^{\mathbb{Z}^2}$.
\end{itemize}
We suppose that the three families above are independent of each other and thus they have joint distribution given by the product probability measure $\textrm{P} = \textrm{P}_p \times \textrm{P}_{\frac{1}{2}} \times \textrm{P}_{q(\cdot)}$ on the product space $\{0,1\}^{\mathbb{Z}^2} \times \{0,1\}^{\mathbb{Z}^2} \times \mathbb{N}^{\mathbb{Z}^2}$. The expectation induced by $\textrm{P}$ will be denoted by $\textrm{E}$. 

\begin{remark}
1. Some coupling arguments used in the proofs ahead will require enlargements of the probability space, we will keep the same notation since the formal specifications of these couplings are straightforward from their descriptions. 2. At some points, we will only be interested in the distributions of the sequences $(\omega((z_1,z_2)))_{z_1 \in \mathbb{Z}}$, $(\theta ((z_1,z_2)))_{z_1 \in \mathbb{Z}}$ and $(\zeta((z_1,z_2)))_{z_1 \in \mathbb{Z}}$ which do not depend on $z_2 \in \mathbb{Z}$, justifying the following abuse of notation $\omega(z_2) = \omega((1,z_2))$, $\theta(z_2) = \theta((1,z_2))$ and $\zeta(z_2) := \zeta((1,z_2))$ for $z_2 \in \mathbb{Z}$.
\end{remark}

\smallskip

For points on the lattice $\mathbb{Z}^2$ we say that $(\tilde{z}_1,\tilde{z}_2)$ is above $(z_1,z_2)$ if $z_2 < \tilde{z}_2$, immediately above if $z_2 = \tilde{z}_2 - 1$, at the right of $(z_1,z_2)$ if $z_1 < \tilde{z}_1$ and at the left if $z_1 > \tilde{z}_1$. Moreover, we say that a site $z \in \mathbb{Z}^2$ is open if $\omega(z) =1$ and closed otherwise. 

Let $h: \mathbb{Z}^2 \times \mathbb{N} \times \{0,1\} \rightarrow \mathbb{Z}^2$ be defined by putting $h(z,k,i)$ as the $k$-th closest open site to $z$ which is immediately above it, when it is uniquely defined, otherwise $h(z,k,0)$ is the $k$-th closest open site to $z$ which is immediately above and at the left of $z$ and $h(z,k,1)$ is the $k$-th closest open site to $z$ which is immediately above and at the right of $z$. Now for every $z \in \mathbb{Z}^2$ define recursively the random path $\Gamma^z(0) = z$ and 
$$
\Gamma^z(n) = ( \Gamma^z_1(n) , \Gamma^z_2(n) ) = h(\Gamma^z(n-1),\zeta(\Gamma^z(n-1)),\theta(\Gamma^z(n-1))),
$$ 
for every $n \ge 1$.
\smallskip

Let $\mathcal{G} = ( V , \mathcal{E} )$ be the random directed graph with set of vertices $V = \mathbb{Z}^2$ and edges $\mathcal{E} = \{ (z,\Gamma^z(1)) : z \in \mathbb{Z}^2 \}$. This generalizes the two-dimensional drainage network model introduced in \cite{grs} which corresponds to the case $q(1)=1$. It will be called here the two-dimensional generalized drainage network model (GDNM). 

The random graph $\mathcal{G}$ may be seen as a set of continuous random paths $X^z = \{ X^z_s : s \ge z_2 \}$. First, define $Z^z_n = \Gamma^z_1(n)$, for every $z\in \mathbb{Z}^2$ and $n\ge 0$. The random path $X^z$ defined by linear interpolation as
$$
X^z_s = (z_2+n+1)-s) Z^z_n + (s-(z_2+n)) Z^z_{n+1} \, ,
$$ 
for every $s \in [z_2+n,z_2+n+1)$ and $n\ge 1$. Put
$$
\mathcal{X} = \{ ( \sigma^{-1} X^z_s,s)_{s\ge z_2} : z \in \mathbb{Z}^2\} \, ,
$$
where $\sigma^2$ is the variance of $X^{(0,0)}_1$ (at the end of the section we show that $\sigma^2$ is finite if $q(\cdot)$ has finite second moment).
The random set of paths $\mathcal{X}$ is called the generalized drainage network. Let
$$
\mathcal{X}_\delta = \{ (\delta x_1 , \delta^2 x_2) \in \mathbb{R}^2 : (x_1,x_2) \in \mathcal{X} \} \, ,
$$
for $\delta \in (0,1]$, be the diffusively rescaled generalized drainage network. 
Note that two paths in $\mathcal{X}$ will coalesce when they meet each other at an integer time. Therefore, if $q(1)<1$, we have a system of coalescing random walks that can cross each other. Our aim is to prove that $\mathcal{X}_\delta$ converges in distribution to $\mathcal{W}$. This extends the results in \cite{cfd} where the case $q(1)=1$ was considered and no crossing between random walks occurs. 

\medskip

We say that $q(\cdot)$ has finite range if there exists a finite set $F \subset \mathbb{N}$ such that $q(n) \neq 0$ if and only if $n \in F$.  Our main result is the following:

\begin{theorem}
\label{thm:conv}
If $q(\cdot)$ has finite range, then the diffusively rescaled GDNM, $\mathcal{X}_\delta$, converges in distribution to the $\mathcal{W}$ as $\delta \rightarrow 0$.
\end{theorem}

\begin{remark}
1. Note that the hypothesis that $q(\cdot)$ has finite range does not imply that the random paths in the GDNM have finite range jumps. It should be clear from the definition that any sufficiently large size of jump is allowed with positive probability. 2. We only need the hypothesis that $q(\cdot)$ has finite range at a specific point in the proof of Theorem \ref{thm:conv}(it will be discussed later), other arguments just require that $q(\cdot)$ has at most  finite absolute fifth moment. So we conjecture that Theorem \ref{thm:conv} holds under this condition.
\end{remark}

\medskip

The next section is devoted to the proof of Theorem \ref{thm:conv}. We finish this section with some considerations about the random paths in the drainage network. The processes $(Z^z_n)_{n \ge 0}$, $z \in \mathbb{Z}^2$, are irreducible aperiodic symmetric random walks identically distributed up to a translation of the starting point. Now consider $X = (X_t)_{\{t\ge 0\}}$ as a random walk with the same distribution of $X^{(0,0)}$.
We show that if $q(\cdot)$ has finite absolute m-th moment then $X_s$ also have this property. 

\medskip

\begin{proposition}
\label{prop:finmom}
If $q(\cdot)$ has finite absolute m-th moment then $X_s$ also has finite absolute m-th moment.
\end{proposition}

\begin{proof}
It is enough to show that the increments of $Z_n := Z^{(0,0)}_n$ have finite absolute m-th moment. Write 
$$
\xi_n = Z_n - Z_{n-1}
$$
for the increments of $Z$. Then $(\xi_n)_{n \ge 1}$ is a sequence of iid random variables with symmetric probability function $\tilde{q}(\cdot)$ such that $\tilde{q}(z)$ is equal to
\begin{eqnarray}
\frac{1}{2} \sum_{k=1}^{2z} \, q(k) \, \textrm{P} \Big( \sum_{j=-z+1}^{z-1} \omega(z) = k-1 , \, \omega(z)( 1 - \omega(-z)) = 1 \textrm{ or } \omega(-z)( 1 - \omega(z)) = 1 \Big) \nn \\
 + \sum_{k=2}^{2z+1} \, q(k) \, \textrm{P} \Big( \sum_{j=-z+1}^{z-1} \omega(z) = k-2 , \, \omega(z)= 1 \textrm{ and } \omega(-z) = 1 \Big) \nn
\end{eqnarray}
which is equal to
\begin{eqnarray}
& 2 p (1-p)^{2z} q(1) + p^{2z+1} q(2z+1) + \qquad \qquad \qquad \qquad \qquad &  \nn \\ \qquad \qquad & + \sum_{k=2}^{2z} p^k (1-p)^{2z-k+1}
\Big[ 2 \Big( \!\! \begin{array}{c} 2z-1 \\ k-1 \end{array} \!\! \Big) + \Big( \!\! \begin{array}{c} 2z-1 \\ k-2 \end{array} \!\! \Big)  \Big] q(k) & \nn
\end{eqnarray}
Therefore, since 
$$
2 \Big( \!\! \begin{array}{c} 2z-1 \\ k-1 \end{array} \!\! \Big) + \Big( \!\! \begin{array}{c} 2z-1 \\ k-2 \end{array} \!\! \Big) \le
\Big( \!\! \begin{array}{c} 2z+1 \\ k \end{array} \!\! \Big)
$$
the absolute m-th moment of $\tilde{q}$ is bounded above by
$$
\sum_{z=1}^{+\oo} z^m \sum_{k=1}^{2z+1} q(z) \Big( \!\! \begin{array}{c} 2z+1 \\ k \end{array} \!\! \Big) p^k (1-p)^{2z -k+1}
$$
which is equal to
$$
\frac{1}{p} \sum_{k=1}^{+\oo} q(k) \sum_{z=\left\lfloor k/2 \right\rfloor}^{+\oo} z^m \Big( \!\! \begin{array}{c} 2z+1 \\ k \end{array} \!\! \Big) p^{k+1} (1-p)^{(2z+2)-(k+1)}
$$
the sum over $z$ is bounded by one half of the absolute m-th moment of a negative binomial random variable with parameters $k+1$ and $p$ which is a polynomial of degree $m$ on $k$. Therefore  
$$
\sum_{z=1}^{+\oo} z^m \tilde{q}(z) \le \sum_{k=1}^{+\oo} r(k) q(k) \, ,
$$
where $r$ is a polynomial of degree $m$. If $q(\cdot)$ has finite m-th moment then the right hand side above is finite, which means that $\tilde{q}(\cdot)$ also has finite m-th moment. \end{proof}

\medskip

\section{Convergence to the Brownian Web} 
\label{sec:convergence}

The proof of Theorem \ref{thm:conv} follows from the verification of four conditions given in \cite{finr} for convergence to the Brownian Web of a system of coalescing random paths that can cross each other. These conditions are described below:

\begin{itemize}
\item[(I)] Let $\mathcal{D}$ be a countable dense set in $\mathbb{R}^2$. For every $y \in \mathbb{R}^2$, there exist single random paths $\theta^y_\delta \in \mathcal{X}_\delta$, $\delta> 0$, such that for any deterministic $y_1,...,y_m \in \mathcal{D}$, $\theta^{y_1}_\delta, ... , \theta^{y_n}_\delta$ converge in distribution as $\delta \ra 0+$ to coalescing Brownian motions, with diffusion constant $\sigma$ depending on $p$, starting at $y_1,...,y_m$. 
\item[(T)] Let $\Gamma_{L,T} = \{ z \in \mathbb{Z}^2 : z_1 \in [-L,L] \textrm{ and } z_2 \in [-T,T] \}$ and $R_{(x_0,t_0)} (u,t)= \{ z \in \mathbb{Z}^2: z_1 \in [x_0 - u , x_0 + u] \textrm{ and } z_2 \in [t_0, t_0+t]\}$. Define $A_{(x_0,t_0)} (u,t)$ as the event that there exists a path starting in the rectangle $R_{(x_0,t_0)} (u,t)$ that exits the larger rectangle $R_{(x_0,t_0)} (Cu,2t)$ through one of its sides, where $C>1$ is a fixed constant. Then 
$$
\limsup_{t \ra 0+} \frac{1}{t} \limsup_{\delta \ra 0+} \sup_{ (x_0,t_0) \in \Gamma_{L,T} } \textrm{P} \big( A_{\mathcal{X}_\delta} (x_0,t_0;u,t) \big) = 0
$$
\item[($B_1^\prime$)] Let $\eta_{\mathcal{V}}(t_0,t;a,b)$, $a<b$, denote the random variable that counts the number of distinct points in $\mathbb{R}\times \{t_0+t\}$ that are touched  by paths in $\mathcal{V}$ which also touch some point in $[a,b]\times \{t_0\}$. Then, for every $\beta > 0$,
$$
\limsup_{\epsilon \ra 0+} \limsup_{\delta \ra 0+} \sup_{ t > \beta } \sup_{ (a,t_0) \in \mathbb{R}^2 } \textrm{P} ( \eta_{\mathcal{X}_\delta}(t_0,t;a,a+\epsilon) \ge 2) = 0 \, .
$$
\item[($B_2^\prime$)] For every $\beta > 0$,
$$
\limsup_{\epsilon \ra 0+} \frac{1}{\epsilon} \limsup_{\delta \ra 0+} \sup_{ t > \beta } \sup_{ (a,t_0) \in \mathbb{R}^2 } \textrm{P} ( \eta_{\mathcal{X}_\delta} (t_0,t;a,a+\epsilon) \ge 3) = 0 \, .
$$
for every $\beta > 0$. 
\end{itemize}

\bigskip

In \cite{nrs} a condition is derived to replace $(B^\prime_2)$ in the proof of convergence of diffusively rescaled non-nearest neighbor coalescing random walks to the Brownian web. Their condition is based on duality and can be expressed as
\begin{itemize}
\item[(E)] Let $\hat{\eta}_\mathcal{V} (t_0,t;a,b)$, $a<b$, be the number distinct points in $(a,b)\times \{t_0+t\}$ that are touched by a path that also touches $\mathbb{R} \times \{t_0\}$. Then for any subsequential limit $\mathcal{X}$ of $\mathcal{X}_\delta$ we have that that
$$
\textrm{E}[ \hat{\eta}_\mathcal{X} (t_0,t;a,b) ] \le \frac{b-a}{\sqrt{\pi t}} \, . 
$$   
\end{itemize}

\bigskip

If we have proper estimate on the distributions of the coalescing times and condition I, then the proof of conditions $B^\prime_1$, T and E follows from adaptations of arguments presented in \cite{nrs}, see also \cite{s}. The estimate we need on the distribution of the coalescing time of two random walks in the drainage network starting at time zero is that the probability that they coalesce after time $t$ is of order $1/\sqrt{t}$. This is central in arguments related to estimates on counting variables associated to the density of paths in the system of coalescing random walks under consideration. Let us be more precise. For each $(z,w) \in \mathbb{Z}^2$, let $Y^{z,w}$ be the random path 
$$
Y^{z,w}_s = X^z_s - X^w_s \, , \ s\ge \max ( z_2 , w_2 ) \, .
$$
The processes $(Y^{z,w})_{(z,w) \in \mathbb{Z}^2}$ are also spatially inhomogeneous random walks with mean zero square integrable increments, identically distributed up to a translation of $z$ and $w$ by the same amount. For $k \in \mathbb{Z}$, let $Y^k = (Y^k_t)_{\{t\ge 0\}}$ be random walk with the same distribution of $Y^{(0,0),(0,k)}$. Define the random times
$$
\tau_k := \min \{ t\ge 0, \ t \in \mathbb{N} : Y^k_t = 0 \} 
$$
and
$$
\nu_k(u) := \min \{ t\ge 0, \ t \in \mathbb{N} : Y^k_t \ge u \} \, ,
$$
for every $k \in \mathbb{Z}$ and $u>0$.

\bigskip

\begin{proposition} 
\label{prop:coaltime1}
If $q(\cdot)$ has finite range, there exists a constant $C>0$, such that 
\begin{equation}
\label{eqprop:coaltime}
\textrm{P} (\tau_k > t) \le \frac{C \, |k|}{\sqrt{t}} \, , \, \textrm{ for every } t>0 \textrm{ and } k \in \mathbb{Z}.
\end{equation}
\end{proposition}

Proposition \ref{prop:coaltime1} has some consequences that we now state:

\begin{lemma} 
\label{prop:coaltime2}
If (\ref{eqprop:coaltime}) holds and $q(\cdot)$ has finite absolute third moment, then for every $u>0$ and $t>0$, there exists a constant $C=C(t,u)>0$, such that
$$
\limsup_{\delta \ra 0} \frac{1}{\delta} \, \sup_{k \le 1} \textrm{P} \Big( \nu_k (\delta^{-1} u) < \tau_k \wedge \delta^{-2} t \Big) < C  \, .
$$
\end{lemma}

\medskip

\begin{lemma}\label{density}
Let $\mathcal{X}_0 = \{(X_s^{(z_1,0)},s):z_1 \in \mathbb{Z}, s \geq 0\}$ be a system of coalescing random walks starting from every point of $\mathbb{Z}$ at time $0$ with increments distributed as the increments of $Z$. Denote by $O_t$ the event that there is a walker seated at the origin by time $t$. If (\ref{eqprop:coaltime}) holds, then
$$
P(O_t) \le \frac{C}{\sqrt{t}}
$$
for some positive constant $C$ independent of everything else.
\end{lemma}

\medskip

\begin{remark}
The Proposition \ref{prop:coaltime1}, Lemma \ref{prop:coaltime2} and Lemma \ref{density} are respectively versions of Lemmas 2.2, 2.4 and 2.7 in section 2 of \cite{nrs}.  Their results hold for the difference of two independent random walks what is not our case. The Proposition \ref{prop:coaltime1} is by itself a remarkable result in our case and is going to be proved in Section \ref{subsec:coaltimes}.
In the proof, we need $q(\cdot)$ with finite range due to an estimate to bound the probability that two walks in the GDNM coalesce given they cross each other (item (iii) of Lemma \ref{lemma:coupling}). It is the only place where the finite range hypothesis is used. Both Lemmas follow from inequality (\ref{eqprop:coaltime}) in Proposition \ref{prop:coaltime1} following arguments presented in \cite{nrs}, although we also prove Lemma \ref{prop:coaltime2} in Section \ref{subsec:coaltimes} to make clear that it holds in our case. Concerning the proof of Lemma \ref{density}, it is exactly the same as the proof of Lemma 2.7 in \cite{nrs}.
\end{remark}

\medskip

Now let us return to the analysis of conditions $B^\prime_1$, T and E.
The idea we carry on here is that if convergence to the BW holds for a system with non-crossing paths then the proof of \cite{nrs} can be adapted to similar systems allowing crossing paths if we have Proposition \ref{prop:coaltime1} and condition $I$. In this direction, \cite{nrs} is extremly useful since it is carefully written in a way that for each result one can clearly identify which hypotheses are needed to generalize its proof to other systems of coalescing random walks.

\medskip

The proof of $B^\prime_1$ is the same as that of section 3 in \cite{nrs} for systems of coalescing random walks independent before coalescence. 
In order to reproduce the proof given in \cite{nrs} we need finite second moment of $X_s$ and that the coalescing time of two walkers starting at distance $1$ apart is of order $1/\sqrt{t}$ which is the content of proposition \ref{prop:coaltime1}.

\medskip

The proof of condition T is also the same as that of section 4 in \cite{nrs} for systems of coalescing random walks independent before coalescence. The proof depends only on two facts: first, an upper bound for the probability of the time it takes for two random walks to get far away one from each other before they coalesce, which is the content of Lemma \ref{prop:coaltime2}; second, finite fifth moment of the increments of the random walks, which follows from the finite fifth moment of $q(\cdot)$.

\medskip

Now we consider condition $E$, but first we introduce more notation. For an $\left(\mathcal{H},\mathcal{F}_{\mathcal{H}}\right)-$valued random variable $\mathbb{X}$, define $\mathbb{X}^{t^-}$ to be the subset of paths in $\mathbb{X}$ which start before or at time $t$. Also denote by $\mathbb{X}(t) \subset \mathbb{R}$ the set of values at time $t$ of all paths in $\mathbb{X}$. In \cite{nrs} it is shown that condition $E$ is implied by condition $E^\prime$ stated below:
\begin{itemize}
\item[($E^{\prime}$)] If $\mathcal{Z}_{t_0}$ is any subsequential limit of $\{\mathcal{X}_{\delta_n}^{t_0^-}\}$ for any $t_0 \in \mathbb{R},$ for any non-negative sequence $(\delta_n)_n$ going to zero as $n$ goes to infinity, then $\forall \ t, a, b \in \mathbb{R}$ with $t > 0$ and $a < b, E[\hat{\eta}_{\mathcal{Z}_{t_0}}(t_0,t;a,b)] \le E[\hat{\eta}_{\mathcal{W}}(t_0,t;a,b)] = \frac{b-a}{\sqrt{\pi t}}.$
\end{itemize}
The verification of condition $E^{\prime}$ follows using the same arguments presented in Section 6 of \cite{nrs}.
Denote by $\mathcal{Z}_{t_0}^{(t_0 + \epsilon)_T}$ the subset of paths in $\mathcal{Z}_{t_0}$ which start before or at time $t_0$ and truncated before time $(t_0 + \epsilon)$. By Lemma \ref{density} we get the locally finiteness of $\mathcal{Z}_{t_0}(t_0 + \epsilon)$ for any $\epsilon > 0$ (for more details see \cite{finr1}, \cite{nrs}). The key point in the verification of condition $E^{\prime}$ is to note that for 
any $0 < \epsilon < t, \ E[\hat{\eta}_{\mathcal{Z}_{t_0}}(t_0,t;a,b)] \le E[\hat{\eta}_{\mathcal{Z}_{t_0}^{(t_0 + \epsilon)_T}}(t_0 + \epsilon,t - \epsilon;a,b)]$. Then, by the locally finiteness of $\mathcal{Z}_{t_0} (t_0 + \epsilon)$ and condition $I$ we get that $\mathcal{Z}_{t_0}^{(t_0 + \epsilon)_T}$ is distributed as coalescing Brownian motions starting from the random set $\mathcal{Z}_{t_0}(t_0 + \epsilon)$. We get condition $E^{\prime}$ from the fact that a family of coalescing Brownian motions  starting from a random locally finite subset of the real line is stochastically dominated by the family of paths of coalescing Brownian motions starting at every point in $\mathbb{R}$ at time $t_0 + \epsilon$ and taking limit when $\epsilon$ goes to zero. See Lemma 6.3 in \cite{nrs} that also holds in our case.

\medskip

\subsection{Estimates on the distributions of the coalescing times} 
\label{subsec:coaltimes}

We start this section proving Lemma \ref{prop:coaltime2} assuming that Proposition \ref{prop:coaltime1} holds. 
We are going to follow closely the proof of proposition 2.4 in \cite{nrs} aiming at making clear that their
arguments holds in our case. Their first estimate is
$$
\textrm{P} \Big( \nu_k (\delta^{-1} u) < \tau_k \wedge ( \delta^{-2} t ) \Big) < C^\prime(t,u) \, |k| \, \delta
$$
holds for $\delta$ sufficiently small. It follows from the following estimate based on the strong Markov Property,  which in 
our case also holds by the independence of the increments,
$$
\textrm{P} (\tau_k > t) \ge \, \textrm{P} \Big( \nu_k (\delta^{-1} u) < \tau_k \wedge ( \delta^{-2} t ) \Big) \,
\inf_{l \in \mathbb{Z}} \textrm{P} \Big( Y^l \in \textrm{B}^l (\delta^{-1}u,\delta^{-2}t) \Big)  \, ,
$$
where $\textrm{B}^l (x,t)$ is the set of trajectories that remain in the interval $[l-x,l+x]$ during the time interval $[0,t]$.
For every $u>0$ and $t>0$, write
$$
\inf_{l \in \mathbb{Z}} \textrm{P} \Big( Y^l \in \textrm{B}^l (\delta^{-1}u,\delta^{-2}t) \Big) =
1 - \sup_{l \in \mathbb{Z}} \textrm{P} \Big( \sup_{s \le \delta^{-2} t} |Y^l_s - l| > \delta^{-1} u \Big) \, .
$$
Then
\begin{eqnarray*}
\limsup_{\delta \ra 0}  \sup_{l \in \mathbb{Z}} \textrm{P} \Big( \sup_{s \le \delta^{-2} t} |Y^l_s - l| > \delta^{-1} u \Big)
& \le & \limsup_{\delta \ra 0} \textrm{P} \Big( \sup_{s \le \delta^{-2} t} ( |X^0_s| + |X^l_s| ) > \delta^{-1} u \Big) \\
& \le & 2 \limsup_{\delta \ra 0} \textrm{P} \Big( \sup_{s \le \delta^{-2} t} |X^0_s|  > \frac{\delta^{-1} u}{2} \Big) \\
& \le & 4 \, \textrm{P} \Big( N > \frac{u}{2\sqrt{t}} \Big)  \le 4 \, e^{-\frac{u^2}{8t}} \, .
\end{eqnarray*}
where $N$ is a standard normal random variable, the last inequality is a consequence of Donsker Invariance Principle, see 
Lemma 2.3 in \cite{nrs}. Therefore we have that Lemma \ref{prop:coaltime2}.

\medskip

The remain of the section is entirely devoted to the proof of Proposition \ref{prop:coaltime1}. It suffices to consider the case in which $k =1$, see the proof of Lemma 2.2 in \cite{nrs}. Here we will considering that $q(\cdot)$ has finite second moment except near the end of the section, in the proof of condition (iii) in Lemma \ref{lemma:coupling}, where we need $q(\cdot)$ with finite range. The proof is based on a proper representation of the process $(Y^1_n)_{n\ge 1}$ which has already been used in \cite{cfd}. However, in our case, the possibility to have $Y_n^1$ assuming negative values, due to crossing paths, requires a new approach. To simplify notation, we write $\tau = \tau_1$. 

\smallskip

By Skorohod Representation Theorem, see section 7.6 in \cite{Du}, there exist a standard Brownian Motion $(B(s))_{s\ge 0}$ and stopping times $T_1$, $T_2$, ..., such that $B(T_n)$ has the same distribution of $Y^1_n$, for $n \ge 0$, where $T_0 = 0$. Furthermore,
the stopping times $T_1$, $T_2$, ..., have the following representation:
$$
T_n = \inf \big\{ s \ge T_{n-1} : B(s) - B(T_{n-1}) \notin \big( U_n(B(T_{n-1})),V_n(B(T_{n-1})) \big) \big\} \, ,
$$
where $\{(U_n(m),V_n(m)): n \ge 1, \, m \in \mathrm{Z} \}$ is a family of independent random vectors taking values in
$\{(0,0)\} \cup \{...,-2,-1\} \times \{1,2,...\}$. Also, by the definition of $Z$, we have that for each fixed $m \in \mathbb{Z}$, $(U_n(m),V_n(m))$, $n \ge 1$, are identically distributed. Indeed, write
$$
\xi^z_n = Z^z_n - Z^z_{n-1}
$$
for the increments of $Z^z$, then $\xi^{(z_1,z_2)}_n = \xi^{(z_1,z_2+(n-1))}_1$ and the families of identically distributed (but dependent) random variables $(\xi^{(z_1,z_2)}_1)_{z_1 \in \mathbb{Z}}$, $z_2 \in \mathbb{Z}$, are independent of each other. 
Moreover
$$
|U_1(m)| \le |\xi^0_1| + |\xi^m_1| \, .
$$
Therefore, by Proposition \ref {prop:finmom}, $|U_1(m)|$ has finite second moment uniformly
bounded in $m \in \mathbb{Z}$. Also note that, by symmetry, for $m > 0$, $|U_1(m)|$ has the same 
distribution as $V_1(-m)$. 

\smallskip

Let us give a brief description of the idea behind the proof of Proposition \ref{prop:coaltime1}. By similar arguments as those of \cite{cfd} we show that $P(T_t \le \zeta t)$ is of order $1/\sqrt{t}$, for some $\zeta > 0$ sufficiently small. On the other hand if $T_t \ge \zeta t$, then for $t$ large, we show that $\zeta t$ is enough time for the Brownian excursions to guarantee that $(Y_n^1)_{0\le n \le t}$ have sufficient sign changes such that, under these changes, the probability that $(Y_n^1)_{0\le n \le t}$ does not hit $0$ is of order $1/\sqrt{t}$.  

\medskip 

Now define the random discrete times $a_0 = 0$, 
$$
a_1 = \inf\{ n \ge 1: Y^1_n \le 0 \}
$$
and for $j \ge 2$
$$
a_j = \left\{
\begin{array}{ll}
\inf\{ n \ge a_{j-1} : Y^1_{n} \ge 0 \} , \ j \textrm{ even}, \\
\inf\{ n \ge a_{j-1} : Y^1_{n} \le 0 \} , \ j \textrm{ odd}.
\end{array}
\right.
$$
Then $(a_j)_{j\ge 1}$ is a increasing sequence of stopping times for the random walk $(Y^1_n)_{n\ge 1}$, which is strictly increasing until $j$ such that $a_j = \tau$. Note that $V_{a_j} ( Y^1_{a_{j-1}} ) \ge - Y^1_{a_{j-1}}$ if $j$ is even and $U_{a_j} ( Y^1_{a_{j-1}} ) \le - Y^1_{a_{j-1}}$ if $j$ is odd. 
The sequence $(a_j)_{j\ge 1}$ induce the increasing sequence of continuous random times $(T_{a_j})_{j\ge 1}$.

\smallskip


We have the inequality
\begin{eqnarray} \label{ineqct}
\textrm{P} ( \tau > t) &\le& \textrm{P}( T_t < \zeta t) + \textrm{P}( \tau > t \, , \, T_t \ge \zeta t ) \nonumber \\
&\le&  \textrm{P}( T_t \le \zeta t) + \sum_{n=1}^{t} 
\textrm{P}( \tau > t \, , \,  T_t \ge \zeta t \, , \, T_{a_{l-1}} < \zeta t \, , \, T_{a_{l}} \ge \zeta t ) \, .
\end{eqnarray}
where $a_0 := 0$. We are going to show that both probabilities in the rightmost side of (\ref{ineqct}) are of order $1/\sqrt{t}$. During the proof we will need some technical lemmas whose proofs are postponed to the end of this section.

\medskip

We start with $\textrm{P}( T_t \le \zeta t)$. Note that we can write
$$
T_{t}= \displaystyle\sum_{i=1}^{t} ( T_i - T_{i-1} ) = \displaystyle\sum_{i=1}^{t}S_{i}(Y^1_{i-1}),
$$
where $(S_{i}(k),\;i\geq1,\;k\in \mathbb{Z})$ are independent random variables.
The sequences $\;(S_{i}(k),\;i\geq1)$, indexed by $k\in \mathbb{Z}$, are not identically distributed.
Fix $\lambda > 0$. By the Markov inequality we have
$$
\textrm{P}( T_t \le \zeta t) = \textrm{P}(e^{- \lambda T_t} \ge e^{- \lambda \zeta t}) \le e^{\lambda \zeta t} \textrm{E}\left(e^{- \lambda T_t}\right).
$$
As in \cite{cfd}, put $S(m) = S_1(m)$, then we have that
\begin{equation}
\label{eq:zeta1}
\textrm{P}( T_t \le \zeta t) \le \left[e^{\lambda \zeta} \sup_{m \in \mathbb{Z}} \textrm{E}\left(e^{-\lambda S(m)}\right)\right]^t .
\end{equation}
To estimate the expectation above, we need some uniform estimates on the distribution of $U_1(m)$ and $V_1(m)$. 

\begin{lemma} \label{p00}
For every $p<1$ and every probability function $q:\mathbb{N} \ra [0,1]$, we have that 
\begin{equation}
\label{eq:c_1}
0< c_1 := \sup_m P( (U_1(m),V_1(m)) = (0,0) ) < 1.
\end{equation}
\end{lemma}

\smallskip

Let $S_{-1,1}$ be the exit time of interval $(-1,1)$ by a standard Brownian motion. We have that
\begin{eqnarray}
\label{eq:zeta2}
\textrm{E}\left(e^{-\lambda S(m)}\right) & = & \textrm{E}\left(e^{-\lambda S(m)} \big| (U(m),V(m)) \neq (0,0) \right) \textrm{P}\left((U(m),V(m)) \neq (0,0)\right) \nonumber \\
& & + \ \textrm{P}\left((U(m),V(m)) = (0,0)\right) \nonumber \\
& \le & \textrm{E}\left(e^{-\lambda S_{-1,1}} \right) \left(1-\textrm{P}\left((U(m),V(m)) = (0,0)\right)\right) \nonumber \\
& & + \ \textrm{P}\left((U(m),V(m)) = (0,0)\right) \nonumber \\
& = & \textrm{P}\left((U(m),V(m)) = (0,0)\right)(1-c_2) + c_2 \nonumber \\
& \le & c_1 (1-c_2) + c_2 
\end{eqnarray}
where $c_1 < 1$ is given by Lemma \ref{p00} and $c_2 = \textrm{E}\left(e^{-\lambda S_{-1,1}} \right) < 1$. 
Now, choose $\zeta$ such that $c_3 = e^{\lambda \zeta} [ c_1 (1-c_2) + c_2] < 1$, then from (\ref{eq:zeta1}) and (\ref{eq:zeta2})
we have 
$$
\textrm{P}( T_t \le \zeta t) \le c_3^t, 
$$
where $c_3 > 0$. Finally, choose $c_4 > 0$ such that $c_3^t \le \frac{c_4}{\sqrt{t}}$ and
\begin{equation}\label{2}
\textrm{P}( T_t \le \zeta t) \le \frac{c_4}{\sqrt{t}} \, ,
\end{equation}
for $t>0$.

\medskip

It remains to estimate the second term in the rightmost side of (\ref{ineqct}) which is 
\begin{equation}
\label{eq:TJestimate1}
 \sum_{l=1}^{t} \textrm{P}( \tau > t \, , \,  T_t \ge \zeta t \, , \, T_{a_{l-1}} < \zeta t \, , \, T_{a_{l}} \ge \zeta t ) \, .
\end{equation}
We need the following result:

\begin{lemma}
\label{lemma:coupling}
There exist independent square integrable random variables $\tilde{R_0}$, $R_j$, $j \ge 1$,
such that:
\begin{enumerate}
\item[(i)] $\{R_j\}_{j=1}^\oo$ are iid random variables;
\item[(ii)] $R_j | \{ R_j \neq 0 \}$ has the same distribution as $\tilde{R}_0$;
\item[(iii)] $c_5 := P(R_{1} \neq 0) < 1$;
\item[(iv)] $T_{a_j}$ is stochastically dominated by $J_j$ which is defined as $J_0 = 0$, 
$$
J_1 = \inf \big\{ s \ge 0 : B(s) - B(0) = (-1)^j ( R_{1} + \tilde{R}_0 ) \big\} \, ,
$$
and
$$
J_j = \inf \big\{ s \ge J_{j-1} : B(s) - B(J_{j-1}) = (-1)^j ( R_{j} + R_{j-1} ) \big\} \, , \ j \ge 2 \, ,
$$
where $(B(s))_{s\ge 0}$ is a standard Brownian motion independent from the sequence $\{R_n\}_{n=1}^\oo$;
\item[(v)] $Y_{a_j} \neq 0$ implies $B(J_j) \neq 0$ which is equivalent to $R_{j} \neq 0$ given that $B(0) = \tilde{R}_0$.
\end{enumerate}

\end{lemma}

\medskip

Let $(J_j)_{j\ge 1}$ be as in the statement of Lemma \ref{lemma:coupling}. Since 
$$
\big\{ \tau > t \, , \,  T_t \ge \zeta t \, , \, T_{a_{l-1}} < \zeta t \, , \, T_{a_{l}} \ge \zeta t \big\}
$$
is a subset of
$$
\big\{ Y_{a_j} \neq 0 \textrm{ for } j=1,...,l-1 \, , \, T_{a_{l}} \ge \zeta t \big\} \, ,
$$
by Lemma \ref{lemma:coupling}, we have that (\ref{eq:TJestimate1}) is bounded above by
\begin{eqnarray}
\label{eq:TJestimate2}
& & \sum_{l=1}^{t} \textrm{P}( B(J_j) \neq 0 \textrm{ for } j=1,...,l-1 \, , \, J_l \ge \zeta t ) \nonumber \\
& & \quad = \, \sum_{l=1}^{t} \textrm{P}( R_{j} \neq 0 \textrm{ for } j=1,...,l-1 \, , \, J_l \ge \zeta t ) \, .
\end{eqnarray}
For the right hand side above, write
\begin{eqnarray}
\label{eq:TJestimate3}
\lefteqn{\textrm{P}( R_{j} \neq 0 \textrm{ for } j=1,...,l-1 \, , \, J_l \ge \zeta t ) \, = } \nonumber \\
& & = \, \textrm{P}( R_{j} \neq 0 \textrm{ for } j=1,...,l-1 ) \, \textrm{P}( J_l \ge \zeta t | R_{j} \neq 0 \textrm{ for } j=1,...,l-1 )
\nonumber \\
& & = \, \textrm{P}( R_{1} \neq 0 )^{l-1} \, \textrm{P}( J_l \ge \zeta t | R_{j} \neq 0 \textrm{ for } j=1,...,l-1 ) \nonumber \\
& & = \, c_5^{l-1} \, \textrm{P}( J_l \ge \zeta t | R_{j} \neq 0 \textrm{ for } j=1,...,l-1 ) \, .
\end{eqnarray}
where the last equality follows from the independence of the $R_j$'s. Now put $\tilde{R}_{j} = R_j | \{ R_j \neq 0 \}$ and define
$\tilde{J}_0 = 0$ and
$$
\tilde{J}_j = \inf \big\{ s \ge \tilde{J}_{j-1} : B(s) - B(\tilde{J}_{j-1}) = (-1)^j ( \tilde{R}_{j} + \tilde{R}_{j-1} ) \big\} \, .
$$
Then $\tilde{R}_{j}$ is also a sequence of iid square integrable random variables and, from Lemma \ref{lemma:coupling}, we get that
\begin{equation}
\label{eq:TJestimate4}
\textrm{P}( J_l \ge \zeta t | R_{j} \neq 0 \textrm{ for } j=1,...,l-1 ) = \textrm{P}( \tilde{J}_l \ge \zeta t ) \, .
\end{equation}
To estimate the right hand side in the previous equality, write
$$
W_1 = \tilde{J}_1 \quad \textrm{and} \quad W_j = \tilde{J}_j - \tilde{J}_{j-1} \textrm{ for } j \ge 1 \, .
$$
Then it is easy to verify that $W_j$, $j \ge 1$, are identically distributed random variables which are
not independent. However $\{W_{2j}\}_{j \ge 1}$ and $\{W_{2j-1}\}_{j \ge 1}$ are families of iid random variables.

We need the following estimate on the distribution function of $W_j$: 

\begin{lemma} \label{tailbehavior}
Let $W_1$ be defined as above. Then, there exists a constant $c_6 > 0$ such that for every $x > 0$ we have that
$$
\textrm{P}( W_1 \ge x ) \le \frac{c_6}{\sqrt{x}}.
$$ 
\end{lemma}

\smallskip

By Lemma \ref{tailbehavior}, it follows that
\begin{equation}
\label{eq:TJestimate5}
\textrm{P} \left( \tilde{J}_{l} \ge \zeta t \right) \le \textrm{P} \left( \sum_{j=1}^{l} W_{j} \ge \zeta t \right) 
\le l \, \textrm{P} \left( W_1 \ge \frac{\zeta t}{l} \right)
\le \frac{c_7 \, l^{\frac{3}{2}}}{\sqrt{t}}, 
\end{equation}
where $c_7 = c_6 / \sqrt{\zeta}$. From (\ref{eq:TJestimate2}), (\ref{eq:TJestimate3}), (\ref{eq:TJestimate4}) and (\ref{eq:TJestimate5}),
we have that 
\begin{equation}
\label{eq:TJestimate6}
\sum_{l=1}^{t} \textrm{P}( B(J_j) \neq 0 \textrm{ for } j=1,...,l-1 \, , \, J_l \ge \zeta t ) 
\, \le \, \frac{c_7}{\sqrt{t}} \sum_{l=1}^{+ \oo} c_5^l l^{\frac{3}{2}} \, \le \, \frac{c_8}{\sqrt{t}} \, .
\end{equation}

\medskip

Back to (\ref{ineqct}), using (\ref{2}) and (\ref{eq:TJestimate6}), we have shown that
$$
\textrm{P} ( \tau > t) \le \frac{c_4 + c_8}{\sqrt{t}} 
$$
finishing the proof of Proposition \ref{prop:coaltime1}.

\medskip

\subsubsection{Proofs of the technical lemmas of section \ref{subsec:coaltimes}}

Here we prove Lemmas \ref{p00}, \ref{lemma:coupling} and \ref{tailbehavior}.

\medskip

\noindent \textit{Proof of Lemma \ref{p00}.}
Since $(U_1(m),V_1(m))$, $m\ge 1$ takes values in $\{(0,0)\} \cup \{...,-2,-1\} \times \{1,2,...\}$,
then we have that $(U_1(m),V_1(m)) = (0,0)$ if and only if  $Y^m_1 = Y^m_0$, which means that 
$Z^m_1 - Z^0_1 = Z^m_0 - Z^0_0 = m$. 

Fix $r = \inf\{k\ge 1: q(k)>0\}$.  Then 
\begin{eqnarray*}
\lefteqn{\textrm{P}( (U_1(m),V_1(m)) \neq (0,0) ) \ge \textrm{P} ( Z^m_1 < Z^0_1 + m ) }\\ 
& & \ge \textrm{P} ( Z^0_1 = k , Z^m_1 < m + k ) \\
& & \ge \textrm{P}( \zeta(0) = \zeta(m) = k , \omega(1-j) = 0, \ \omega(j)=\omega(m+j-1)=1, \ j=1,...,k )\\
& & \ge \frac{1}{2} q(k)^2 p^{3k} > 0 \, .
\end{eqnarray*}
Since the last term in the previous inequality does not depend on $m$, we have that
$$
\sup_m \textrm{P}( (U_1(m),V_1(m)) = (0,0) ) \le 1 - \frac{1}{2} q(k)^2 p^{3k} < 1.
$$
To prove the other inequality we consider two cases. First, suppose that $m < \lfloor k/2 \rfloor$,
then
\begin{eqnarray*}
\lefteqn{\textrm{P}( (U_1(m),V_1(m)) = (0,0) ) \ge \textrm{P} ( Z^0_1 = k-1 , \, Z^m_1 = m + k -1 ) } \\
& & \qquad \ge \textrm{P} ( \zeta(0) = \zeta(m) = k , \ \omega(j)=0, \ j=m-k+1,...,m-1, \\
& & \qquad \qquad \quad \omega(j-k)=1, \ j=1,...m , \ \omega(m+j) = 1 , \ j=0,...,k ) \\
& & \qquad \ge \frac{1}{2} q(k)^2  p^{m+2k-1} \ge \frac{1}{2} q(k)^2  p^{\frac{5}{2}k} \, .
\end{eqnarray*}
Now we consider the case $m \ge \lfloor k/2 \rfloor$. In this case 
\begin{eqnarray*}
\lefteqn{\textrm{P}( (U_1(m),V_1(m)) = (0,0) ) \ge P ( Z^0_1 = \lfloor k/2 \rfloor , \, Z^m_1 = m + \lfloor k/2 \rfloor \Big) } \\
& & \qquad \ge \textrm{P} ( \zeta(0) = \zeta(m) = k , \ \omega(j) = \omega(m+j) = 1, \ j = -\lfloor k/2 \rfloor ,... , \lfloor k/2 \rfloor ) \\
& & \qquad \ge \frac{1}{4} q(k)^2  p^{2k} \, .
\end{eqnarray*}
Therefore,
$$
\sup_m \textrm{P}( (U_1(m),V_1(m)) = (0,0) ) \ge \min \Big\{ \frac{1}{2} q(k)^2  p^{\frac{5}{2}k} , \frac{1}{4} q(k)^2  p^{2k} \Big\} > 0.
$$
\begin{flushright}
$\square$
\end{flushright}

\medskip

\noindent \textit{Proof of Lemma \ref{lemma:coupling}.} Consider independent standard Brownian motions $(B(s))_{s\ge 0}$ and $(\mathbb{B}(s))_{s\ge 0}$ such that $(B(s))_{s\ge 0}$ is used in the Skorohod representation of $(Y^1_n)_{n\ge 1}$ and $(\mathbb{B}(s))_{s\ge 0}$ is also independent of $(Y^1_n)_{n\ge 1}$. In the time interval $[0,T_{a_1}]$, $(B(s))_{s\ge 0}$ makes an excursion being able to visit $(-\infty,0]$ when $U(Y^1_n) \le -Y^1_n$ for some $1 \le n \le a_1$. When this happens two things may occur:
\begin{enumerate}
\item[(1)] $U(Y^1_n) = - Y^1_n$ which implies $a_1=n$ and $B(T_{a_1}) = Y_{a_1} = 0$ meaning that $Z^0$ and $Z_1$ have coalesced; 
\item[(2)] $U(Y^1_n) < - Y^1_n$ which implies that, with probability greater than some strictly positive constant $\beta$, $a_1=n$ and $(B(s))_{s\ge 0}$ will leave the interval $[U(Y^1_n) + Y^1_n ,V(Y^1_n) + Y^1_n]$ by its left side (strong Markov property with the fact that $(B(s))_{s\ge 0}$ visits $(-\infty,0)$) meaning that $Z^0$ and $Z^1$ have crossed. 
\end{enumerate}

Denote $\mathcal{N}_1$ by the random variable that denotes the cardinality of $E := \big\{ n \in \{1,...,a_1\} : (B(s))_{s\ge 0} \textrm{ visits } (-\infty,0] \textrm{ in time interval } [T_{n-1} , T_{n}] \big\}$. Note that $\mathcal{N}_1 \ge 1$ and, from (2), $\mathcal{N}_1$ is stochastically bounded by a geometric random variable with parameter $\beta$. We will construct below a sequence of iid non positive square integrable random variables $(\mathcal{R}_n)_{n \ge 1}$ such that for each $n \in E$, $\mathcal{R}_n \le U(Y^1_n) + Y^1_n$. Define
$$
R_1 = - \sum_{n=1}^{\mathcal{G}_1} \mathcal{R}_n \, ,
$$
where $\mathcal{G}_1$ is a geometric random variable with parameter $\beta$ such that $\mathcal{G}_1 \ge \mathcal{N}_1$. It is a simple  to show that $R_1$ is square integrable and independent of $(\mathbb{B}(s))_{s\ge 0}$. Clearly from the definitions 
$$
R_1 \ge - \sum_{n=1}^{\mathcal{N}_1} \mathcal{R}_n \ge | Y_{a_1}^n | = | B(T_{a_1}) | \, .
$$
So if we take $\tilde{R}_0$ with the same distribution of $R_1 | \{R_1 \neq 0\}$ then the time for $(B(s))_{s\ge 0}$ starting at $1$ to hit $-R_1$ is stochastically dominated by the time for $(\mathbb{B}(s))_{s\ge 0}$ starting at $\tilde{R}_0$ to hit $-R_1$ which we denote $J_1$. 

\smallskip

From this point, it is straightforward to use an induction argument to build the sequence $\{R_j\}_{j\ge 1}$. At step $j$ in the induction argument, we consider the $(B(s))_{s\ge 0}$ excursion in time interval $[T_{a_{j-1}},T_{a_{j}}]$, and since $| Y_{a_{j-1}}^n | \le R_{j-1}$ we can obtain $R_j$ and define $J_j$ using $(\mathbb{B}(s))_{s\ge 0}$ as before. By the strong Markov property of $(Y_n^1)$, we obtain that the $R_j$'s are independent and properties (i), (ii), (iv) and (v) follows directly from the construction.

\medskip

It remains to prove (iii) and construct $(\mathcal{R}_n)_{n \ge 1}$. 

\bigskip

\noindent \textbf{Construction of $(\mathcal{R}_n)_{n \ge 1}$:} Fix $n \ge 1$. We have to show that there exists a square integrable random variable $\mathcal{R}$ that stochastically dominates $(U_n(m)+m)|\{U_n(m)\le -m\}$ for every $m>0$. From this it is straightforward to obtain the sequence $(\mathcal{R}_n)_{n \ge 1}$, by induction and extension of the probability space, in a way that $\mathcal{R}_n$ has the same distribution as  $\mathcal{R}$. 

\smallskip

Now fix $m>0$. We will be conditioning of the event $Y^1_n = m$. Recall that $Y^1_n = Z^1_n - Z^0_n$, $n\ge 1$. If $(Z^0_n)_{n\ge 1}$ and $(Z^1_n)_{n\ge 1}$ were independent, $(Y^1_n)_{n\ge 1}$ would be spatially homogeneous and $\mathcal{R}$ could be fixed as the sum of the absolute values of two independent increments of the processes $(Z^i_n)_{n\ge 1}$, $i=1,2$. That is not the case. 

We define $\mathcal{R} = 0$ if $U_n(m) = -m$. On the event $\{U_n(m) < -m\}$, $\mathcal{R}$ should be smallest value that the environment at time $n+1$ allows for $U_n(m)+m < 0$. So let us regard the environment at time $n+1$. Consider the partition of the sample space on the events $G_{k,m} = \{ Z^0_{n} = k \} \cap \{Z^1_{n} = m + k\}$, for $m > 0$ and $k \in \mathbb{Z}$. We define random variables
$\eta$, $\tilde{\eta}$ and $\hat{\eta}$ in the following way: On $G_{k,m}$ 
$$
\eta := (-k) + \textrm{ position of the } \zeta(Z^0_{n}) \textrm{ open site at time } n+1 \textrm{ to the right of position } k ,
$$ 
$$
\tilde{\eta} := k - \textrm{ position of the } \zeta(Z^0_{n}) \textrm{ open site at time } n+1 \textrm{ to the left of position } k , 
$$
$$
\hat{\eta} := k - \tilde{\eta} - \textrm{ position of the } \zeta(Z^1_{n}) \textrm{ open site at time } n+1 \textrm{ to the left of position } (k - \tilde{\eta}) ,  
$$
see figure $1$ below. Clearly, conditioned to each $G_{k,m}$, $\eta$, $\tilde{\eta}$ and $\hat{\eta}$ are iid whose distribution is caracterized as the sum of $N$ geometric random variables, where $N$ is distributed according to the probability function $q(\cdot)$ and is independent of the geometric random variables. In particular, $\eta$, $\tilde{\eta}$ and $\hat{\eta}$ do not depend on $k$ and $m$ and
it is straightforward to verify that they have finite second moment if $q(\cdot)$ does. Define $\mathcal{R} = - ( \eta + \tilde{\eta} + \hat{\eta})$. On the event $\{U_n(m) < -m\}$, $U_n(m)+m$ cannot be smaller than $\mathcal{R}$. Indeed on $G_{k,m}$, if the random walks $Z^0$ and $Z^1$ cross each other, we have 
$$
k - \tilde{\eta} - \hat{\eta} \le Z^1_{a_j} < Z^0_{a_j} \le k + \eta \, ,
$$
thus
$$
0 < U_n(m)+m = Y^1_{n+1} = Z^1_{a_j} - Z^0_{a_j} \ge - ( \eta + \tilde{\eta} + \hat{\eta}) = \mathcal{R} \, .
$$

\begin{figure}[!ht]
\begin{center}
\psfrag{A}{$\hat{\eta}$}
\psfrag{B}{$\tilde{\eta}$}
\psfrag{C}{$\eta$}
\psfrag{Z}{$\mathbb{Z}$}
\psfrag{d}{$\ldots$}
\psfrag{0}{$k=0$}
\psfrag{D}{Closed site}
\psfrag{E}{Open site}
\resizebox{300pt}{!}{\includegraphics{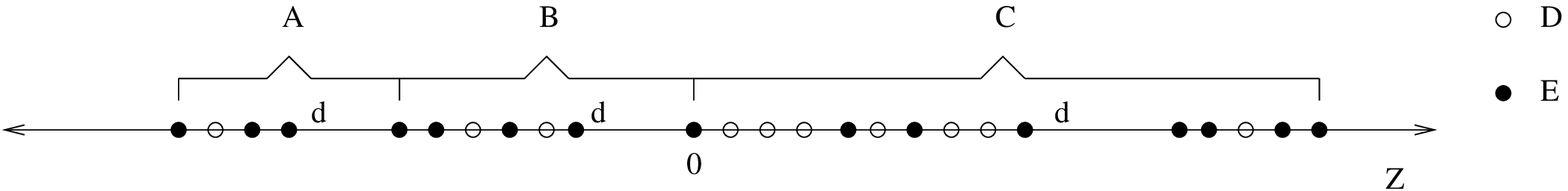}} \label{figure}

Figure $1:$ \ A realization of the random variables $\eta, \tilde{\eta}, \hat{\eta}.$
 \end{center}
\end{figure}

\begin{remark}
In figure $1$ above the interval with size $\eta$ has exactly $\zeta(Z^0_{n})$ open sites, the interval with size $\tilde{\eta}$ has $\zeta(Z^0_{n})$ open sites and finally the interval with size $\hat{\eta}$ has exactly $\zeta(Z^1_{n})$ open sites.
\end{remark}

\bigskip

\noindent \textbf{Proof of (iii):} Conditioned on the event $\{\mathcal{G}_1 = 1\}$, which occurs with probability $\beta$, we have that $B(T_{a_1}) \le 0$ and we can suppose $R_1$ equal in distribution to $\mathcal{R}$ defined above. Thus 
\begin{eqnarray*}
P(R_{1} = 0) & = & \beta \, P( \mathcal{R} = 0 | B(T_{a_1}) \le 0) \\ 
& \ge & \beta \inf_{m>0} P( U_1(m) = -m | U_1(m) \le -m , B(T_1) - B(0) = U_1(m)) \\
& = & \beta \inf_{m>0} P( Z^m_1 - Z^0_1 = 0 | Z^m_1 - Z^0_1 \le 0) \, ,
\end{eqnarray*}
where the first inequality follows from the definition of $\mathcal{R}$. So we have (ii) if
\begin{equation}
\label{eq:cruzaecoal}
\inf_{m>0} P( Z^m_1 - Z^0_1 = 0 | Z^m_1 - Z^0_1 \le 0) > 0 \, .
\end{equation}
Note that the probability above, is the probability that two walks in the GDNM at distance $m$ coalesce after their first jumps given that they cross each other.

\smallskip

Lets prove (\ref{eq:cruzaecoal}). Suppose that $q(\cdot)$ has finite range (Here is the single point in the text where we need $q(\cdot)$ with finite range).  Fix $L \ge 1$ such that $q(k) = 0$, if $k \ge L$. Since clearly $P( Z^m_1 - Z^0_1 = 0) > 0$ for every $m>0$, it is enough to show that 
$$
\inf_{m>4L} P( Z^m_1 - Z^0_1 = 0 | Z^m_1 - Z^0_1 \le 0) > 0 \, .
$$
Then fix $m>L$ and write
\begin{eqnarray}
\label{eq:dec}
P(Z^m_1 < Z^0_1) & \le & P(Z^m_1 < Z^0_1, Z^0_1 < m-L) + P(Z^m_1 < Z^0_1, Z^m_1>L) \nn \\
& & + \, P(Z^m_1 < Z^0_1, Z^m_1 \le L, Z^0_1 \ge m-L)
\end{eqnarray}
where by symmetry the left hand side is equal to
\begin{equation}
\label{eq:dec1}
2 P(Z^m_1 < Z^0_1, Z^0_1 < m-L) + P(Z^m_1 < Z^0_1, Z^m_1 \le L, Z^0_1 \ge m-L) \, .
\end{equation}
We claim that each term in the previous sum is bounded above by $C \, P(Z^m_1 = Z^0_1)$, for some $C>0$ not
depending on $m$, which implies that 
$$
P( Z^m_1 - Z^0_1 = 0 | Z^m_1 - Z^0_1 \le 0) \ge \frac{1}{3C+1}
$$
Let us consider $P(Z^m_1 < Z^0_1, Z^0_1 < m-L)$ first.. If $Z^m_1 < Z^0_1$ and $Z^0_1 < m-L$, we define 
$$
M = \sum_{j=Z^m}^{Z^0-1} \omega(j) \, .
$$ 
Clearly $M < L$ and $\sum_{j=m-L}^{m+L} \omega(k) < L$ and we can choose $j(1)< ... <j(M)) \in \{m-L,...,m+L\}$
such that $\omega(j(i)) = 0$ for $i \in \{1, ... ,M\}$ and if $\omega(j) = 0$ than $j > j(E_{k,l})$.
By changing the occupancies at sites $\{j(1), ... ,j(M)\}$ we have that the set $\{ Z^m_1 < Z^0_1, Z^0_1 < m-L \}$
has probability bounded above by 
$$
\max\Big( \Big(\frac{1-p}{p}\Big)^L , 1 \Big)  P(Z^m_1 = Z^0_1)\, .
$$
Therefore we have proved the claim for the first probability in (\ref{eq:dec1}). 

Now we estimate $P(Z^m_1 < Z^0_1, Z^m_1 \le L, Z^0_1 \ge m-L)$. Since $m>4L$,
if $Z^m_1 < Z^0_1$, $Z^m_1 \le L$ and $Z^0_1 \ge m-L$, then there exists at least $L$ open sites in the interval
$\{L,L+1,...,m-L\}$. By changing the occupancies of at most $L$ of these sites, we get a configuration with $Z^m_1 = Z^0_1$.
Thus, $P(Z^m_1 < Z^0_1, Z^m_1 \le L, Z^0_1 \ge m-L)$ is bounded above by
$$
\max\Big( \Big(\frac{1-p}{p}\Big)^L , 1 \Big)  P(Z^m_1 = Z^0_1) \, .
$$
\begin{flushright}
$\square$
\end{flushright}
\medskip

\noindent \textit{Proof of Lemma \ref{tailbehavior}.} Put $\mu := E[|\xi^0_1|]$ and $\mathcal{J}_m := \inf \{ t \ge 0 : B(t) = m \}$. We have that
\begin{eqnarray*}
\textrm{P}( W_1 \ge x ) & = & \sum_{k \ge 1} \sum_{l \ge 0} \textrm{P}( W_1 \ge x | R_1 = k, \ R_2 = l ) \textrm{P}( R_1 = k, \ R_2 = l ) \\
& = & \sum_{k \ge 1} \sum_{l \ge 0} \textrm{P}( W_1 \ge x | R_1 = k, \ R_2 = l ) \textrm{P}( R_1 = k) \textrm{P}(R_2 = l) \\
& = & \sum_{k \ge 1} \sum_{l \ge 0} \textrm{P}^0 ( \mathcal{J}_{k+l} \ge x ) \tilde{q}(k) \tilde{q}(l) \\
& = & \sum_{k \ge 1} \sum_{l \ge 0} \int_x^\oo \frac{k+l}{\sqrt{2 \pi y^3}} e^{- \frac{k+l}{2y}} dy \, \tilde{q}(k) \tilde{q}(l) \\
& \le & 2 \mu^2 \int_x^\oo \frac{1}{\sqrt{2 \pi y^3}} e^{-\frac{1}{2y}} dy  \le \frac{c_6}{\sqrt{x}} \, .
\end{eqnarray*}
\begin{flushright}
$\square$
\end{flushright}

\medskip
\subsection{Verification of condition I} 
\label{subsec:I}

\bigskip 

Condition $I$ clearly follows from the next result (see Theorem 4 in \cite{cfd} for the equivalent result in the case $q(1)=1$):

\begin{proposition}
\label{prop:coalBM}
Suppose that inequality (\ref{eqprop:coaltime}) holds and $q(\cdot)$ has finite absolute second moment. Let $(x_0,s_0)$, $(x_1,s_1)$, ... , $(x_m,s_m)$ be $m+1$ distinct points in $\mathbb{R}^2$ such that $s_0 \le s_1 \le ... \le s_m$, 
and if $s_{i-1} = s_i$ for some $i$, $i=1,...,m$, then $x_{i-1} < x_i$. Then
$$
\Big( \Big( \frac{X^{x_0 n}_{s n^{2}}}{n} \Big)_{s \ge s_0} , ... , \Big( \frac{X^{x_m n}_{s n^{2}}}{n} \Big)_{s \ge s_m} \Big)_{s \ge 0} \Longrightarrow^D \big( \big( B^{x_0}_0(s) \big)_{s \ge s_0} , ... , \big( B^{x_m}_m(s) \big)_{s \ge s_m} ) \, ,
$$
where $\big( B^{x_0}_0(s) \big)_{s \ge s_0} , ... , \big( B^{x_m}_m(s) \big)_{s \ge s_m}$ are coalescing Brownian Motions with constant diffusion coefficient $\sigma$ starting at $((x_0,s_0),...,(x_m,s_m))$. 
\end{proposition}

To prove Proposition \ref{prop:coalBM}, we first remark that the same proof of part III in the proof of Theorem 4 of \cite{cfd} holds in our case.
So it is enough to consider the case $s_0 = s_1 = ... = s_k = 0$.

\medskip

Let $(B^{x_0}_0(s),B^{x_1}_1(s),...,B^{x_m}_m(s))_{s \ge 0}$ be a vector of coalescing Brownian Motions starting at $(x_0,...,x_m)$ with diffusion coefficient $\sigma$. Under the assumption that $q$ is finite range, we show in this section that
\begin{equation}
\label{eq:convdist1}
\Big( \frac{X^{x_0 n}_{s n^{2}}}{n} , ... , \frac{X^{x_m n}_{s n^{2}}}{n} \Big)_{s \ge 0} \Longrightarrow^D (B^{x_0}_0(s), ... , B^{x_m}_m(s))_{s \ge 0} \, .
\end{equation}
We have that the convergence of $(X^{x n}_{s n^{2}} / n)_{s \ge 0}$ to a Brownian Motion starting at $x$ is a direct consequence of Donsker invariance principle. For the case $q(1)=1$, in the proof presented in \cite{cfd}, it is shown by induction that for $j=1,...,m$
\begin{equation*}
\label{eq:convdist2}
\Big( \frac{X^{x_k n}_{s n^{2}}}{n} \Big)_{s \ge 0} \, \Big| \, \Big( \frac{X^{x_0 n}_{s n^{2}}}{n}, ... , \frac{X^{x_{k-1} n}_{s n^{2}}}{n} \Big)_{s \ge 0} \Longrightarrow^D (B^{x_k}_k(s) )_{s \ge 0} \, | \, ( B^{x_0}_0(s) , ... ,B^{x_{k-1}}_{k-1}(s) )_{s \ge 0} \,.
\end{equation*}
It relies on the fact that no crossing can occur which allows the use 
of a natural order among trajectories of the random walks which we do not have. 

\medskip

We will take a slight different approach here. 
We make the proof in the case $x_0=0$, $x_1=1$, ... , $x_m = m$, 
the other cases can be carried out in an analogous way. 
To simplify the notation write
$$
\frac{X^{0}_{s n^{2}}}{n} = \Delta_n^0(s) \ , \ \  
\frac{X^{n}_{s n^{2}}}{n} = \Delta_n^1(s) \ , \ \ldots \ , \
\frac{X^{m n}_{s n^{2}}}{n} = \Delta_n^m(s) \, . 
$$
Let us fix a uniformly continuous bounded function $H:D([0,s])^{m+1} \ra \mathbb{R}$. We show that
\begin{eqnarray}
\label{I2}
\lim_{n \ra \oo} |\textrm{E}[H(\Delta^{0}_n, \, \ldots \, , \Delta^m_n)] - \textrm{E}[H(B^{0}_0, \, \ldots \, ,B^{m}_m)]| = 0 \, .
\end{eqnarray}
As mentioned above $\Delta^{0}_n$ converges in distribution to $B^{0}_0$.
Now, we suppose that $(\Delta^{0}_n, \, \ldots \, , \Delta^{m-1}_n)$ converges in distribution to 
$(B^{0}_0, \, \ldots \, ,B^{m-1}_{m-1})$ and we are going to show that (\ref{I2}) holds. By induction 
and the definition of convergence in distribution we obtain (\ref{eq:convdist1}).

\medskip

We start defining a modifications of the random walks $X^0$, ... ,$X^{(m-1)n}$, $X^{mn}$ in order that the modification of $X^{mn}$
is independent of the modifications of $X^0$, ... ,$X^{(m-1)n}$, until the time of coalescence with one of them. We achieve this through 
a coupling which is constructed using a suitable change of the environment.

So let $(\bar{\omega}^0(z,t))_{z \in \mathbb{Z}, t \in \mathbb{N}}$ and 
$(\bar{\omega}^1(z,t))_{z \in \mathbb{Z}, t \in \mathbb{N}}$ be two independent families of independent 
Bernoulli random variables with parameter $p \in (0,1)$ which are also independent of 
$(\omega(z,t))_{z \in \mathbb{Z}, t \in \mathbb{N}}$, $(\theta (z,t))_{z \in \mathbb{Z}, t \in \mathbb{N}}$ and 
$(\zeta (z,t))_{z \in \mathbb{Z}, t \in \mathbb{Z}_+}$. Considering the processes $Z^{0}$, ..., $Z^{(m-1)n}$ 
we define the environment $(\tilde{\omega}(z,t))_{z \in \mathbb{Z}, t \in \mathbb{N}}$ by
$$
\tilde{\omega}(z,t) = \left\{
\begin{array}{ll}
\omega(z,t) \, , & \! \textrm{if } \, t\ge 1 , \ z \le \max_{0\le k \le m-1} Z_{t-1}^{kn} + n^{\frac{3}{4}}, \\
\bar{\omega}^0(z,t) \, , & \! \textrm{otherwise}.
\end{array}
\right. 
$$
The processes $\tilde{Z}^{0}$, ..., $\tilde{Z}^{(m-1)n}$ are defined in the following way: For every $k=0,...,m-1$, $\tilde{Z}^{k n}$ starts at $k n$
and evolves in the same way as $Z^{kn}$ except that $\tilde{Z}^{kn}$ sees a different environment. Both $Z^{kn}$ and $\tilde{Z}^{kn}$ processes use the families of random variables $(\theta (z,t))_{z \in \mathbb{Z}, t \in \mathbb{N}}$ and $(\zeta (z,t))_{z \in \mathbb{Z}, t \in \mathbb{Z}_+}$, but $\tilde{Z}^{kn}$ jumps according to $(\tilde{\omega}(z,t))_{z \in \mathbb{Z}, t \in \mathbb{Z}_+}$ until the random time $t$ such that
$$
\max_{0\le k \le m-1} \tilde{Z}_{t}^{kn} \ge \max_{0\le k \le m-1} \tilde{Z}_{t-1}^{kn} + n^{\frac{3}{4}} \, ,
$$
after that time $\tilde{Z}^{kn}$ jumps according to $(\bar{\omega}^0(z,t))_{z \in \mathbb{Z}, t \in \mathbb{N}}$. From $\tilde{Z}^{kn}$ we define $\tilde{X}^{kn}$ and $\tilde{\Delta}^k_n$ as before.

Considering the process $Z^{mn}$ we define the environment 
$(\hat{\omega}(z,t))_{z \in \mathbb{Z}, t \in \mathbb{N}}$ by
$$
\hat{\omega}(z,t) = \left\{
\begin{array}{ll}
\omega(z,t) \, , & \! \textrm{if } \, t\ge 1 , \ Z_{t-1}^{mn} - n^{\frac{3}{4}} \le z \le Z_{t-1}^{mn} + n^{\frac{3}{4}}, \\
\bar{\omega}^1(z,t) \, , & \! \textrm{otherwise}.
\end{array}
\right. 
$$
The process $\hat{Z}^{mn}$ is defined in the following way: $\hat{Z}^{mn}$ starts at $m n$ and evolves in the same way as $Z^{mn}$ exept that $\hat{Z}^{mn}$ sees a different environment. Both $Z^{mn}$ and $\hat{Z}^{mn}$ use the families of random variables $(\theta (z,t))_{z \in \mathbb{Z}, t \in \mathbb{N}}$ and $(\zeta (z,t))_{z \in \mathbb{Z}, t \in \mathbb{Z}_+}$, but $\hat{Z}^{mn}$ jumps according to $(\tilde{\omega}(z,t))_{z \in \mathbb{Z}, t \in \mathbb{Z}_+}$ until the random time $t$ such that
$$
\hat{Z}^{mn} \neq  \big[ \hat{Z}_{t-1}^{mn} - n^{\frac{3}{4}} , \hat{Z}_{t-1}^{mn} + n^{\frac{3}{4}} \big] ,
$$
after that time $\hat{Z}^{mn}$ jumps according to $(\bar{\omega}^1(z,t))_{z \in \mathbb{Z}, t \in \mathbb{N}}$. From $\hat{Z}^{mn}$ we define $\hat{X}^{mn}$ and $\hat{\Delta}^m_n$ as before. 

Define the event
\begin{equation*}
\mathcal{B}_{n,s}  =  \bigcap_{k=0}^m 
\left\{ 
\begin{array}{l}
\Delta^k_n \textrm{ does not make a jump of size greater} \\
\textrm{than } n^{-\frac{1}{4}} \textrm{ in time interval } 0 \le t \le s
\end{array}
\right\} \, .
\end{equation*}
Note three facts:
\begin{itemize}
\item[$\cdot$] $\hat{Z}^{mn}$ is independent of $(\tilde{Z}^0, ...,\tilde{Z}^{(m-1)n})$ when conditioned to the event that $\tilde{Z}^{mn}$ do not get to a distance smaller than $2 n^{3/4}$ of some $\tilde{Z}^{jn}$, $0 \le j < m$. 
\item[$\cdot$] For every $k=0,...,m-1$, $Z^{kn}$ and $\tilde{Z}^{kn}$  are equal at least until the first jump by one of them of size greater than $n^{3/4}$, thus they are equal when retricted to the event $\mathcal{B}_{n,s}$. The same is true about $Z^{mn}$ and $\hat{Z}^{mn}$.
\item[$\cdot$] $(\tilde{\Delta}^{0}_n, \, \ldots \, , \tilde{\Delta}^{m-1}_n)$ and 
$(\Delta^{0}_n, \, \ldots \, , \Delta^{m-1}_n)$ have the same distribution.
\end{itemize} 
We have the following lemma:

\begin{lemma} If $q$ has finite third moment, then
$$
\textrm{P} ( \mathcal{B}^c_{n,s} ) \le C \, n^{-\frac{1}{4}} \, .
$$
\end{lemma}

\begin{proof} By definition, we have
$$
\textrm{P} ( \mathcal{B}^c_{n,s} ) \le ( m + 1 ) \, \textrm{P} \big( (Z^0_j)_{j=1}^{sn^2} \textrm{makes a jump of size greater than } n^{3/4} \big) \, .
$$
It is simple to see that the probability in the right hand side is bounded above by
$$ 
s \, n^2 \textrm{P} \big( |Z^0_1 - Z^0_0| \ge n^{3/4} \big)
\le s \, n^2 \, n^{-\frac{9}{4}} \, \textrm{E}[|Z^0_1 - Z^0_0|^3] \le C \, n^{-\frac{1}{4}} \, .
$$
\end{proof}

We still need to replace the process $\hat{Z}^{mn}$ by a process $\bar{Z}^{mn}$ in such a way that $\hat{Z}^{mn}$ is independent of 
$(\tilde{Z}^0,Z^n,...,\tilde{Z}^{(m-1)n})$ until the time $\bar{Z}^{mn}$ coalesce with one of them. In order to do this we define 
$\bar{Z}^{mn}$ in the following way: Let 
$$
\nu = \inf\big\{1\le j \le s \, n^2 : \, | \tilde{Z}^{k n}_j - \hat{Z}^{m n}_j| \le n^{\frac{3}{4}} \textrm{ for some } 0\le k \le m-1\big\},
$$
then $\bar{Z}^{m n}_j = \hat{Z}^{m n}_j$, for $0\le j < \nu$, otherwise $\bar{Z}^{m n}$ jumps according to $(\bar{\omega}^1(z,t))_{z \in \mathbb{Z}, t \in \mathbb{N}}$ or coalesce with one of $\tilde{Z}^0$, ..., $\tilde{Z}^{(m-1)n}$, if they meet each other at an integer time. 
From $\bar{Z}^{m n}$ we define $\bar{X}^{m n}$ and $\bar{\Delta}^m_n$ as before. Define the event
$$
\mathcal{A}_{n,s} = \Big\{ \inf_{0\le k < m} \inf_{0\le t \le s} \| \tilde{\Delta}^{k}_n(t) - \hat{\Delta}^{m}_n(t) \| \ge 2 \, n^{-\frac{1}{4}} \Big\}  
$$
and note two facts: 
\begin{itemize}
\item[$\cdot$] $\bar{Z}^{mn}$ is independent of $(\tilde{Z}^0,...,\tilde{Z}^{(m-1)n})$. 
\item[$\cdot$] $\hat{Z}^{m n}$ and $\bar{Z}^{m n}$ are equal when retricted to the event $\mathcal{A}_{n,s}$.
\end{itemize}

\medskip
We are now ready to prove (\ref{I2}). Write
\begin{eqnarray}
\label{I1}
\lefteqn{|\textrm{E}[H(\Delta^{0}_n, \, \ldots \, , \Delta^m_n)] - \textrm{E}[H(B^{0}_0, \, \ldots \, ,B^{m}_m)]|  \le } \nn \\
& & \qquad |\textrm{E}[H(\Delta^{0}_n, \, \ldots \, , \Delta^m_n)] - \textrm{E}[H(\tilde{\Delta}^{0}_n, \, \ldots \, , \tilde{\Delta}^{m-1}_n, \hat{\Delta}^m_n)]| \nn \\
& & \qquad \quad + |\textrm{E}[H(\tilde{\Delta}^{0}_n, \, \ldots \, , \tilde{\Delta}^{m-1}_n, \hat{\Delta}^m_n)] - \textrm{E}[H(\tilde{\Delta}^{0}_n, \, \ldots \, , \tilde{\Delta}^{m-1}_n, \bar{\Delta}^m_n)]| \\
& & \qquad \quad \quad + |\textrm{E}[H(\tilde{\Delta}^{0}_n, \, \ldots \, , \tilde{\Delta}^{m-1}_n, \bar{\Delta}^m_n)] - \textrm{E}[H(B^{0}_0, \, \ldots \, ,B^{m}_m)]|. \nn
\end{eqnarray}
By Donsker's Invariance Principle and the induction hypothesis, we have that
$$
\lim_{n\ra \infty} |\textrm{E}[H(\tilde{\Delta}^{0}_n, \, \ldots \, , \tilde{\Delta}^{m-1}_n, \bar{\Delta}^m_n)] - \textrm{E}[H(B^{0}_0, \, \ldots \, ,B^{m}_m)]| = 0 \, .
$$
So we only have to deal with the first and second term in (\ref{I1}). For the first term in (\ref{I1}) we have that
\begin{eqnarray*}
\lefteqn{|\textrm{E}[H(\Delta^{0}_n, \, \ldots \, , \Delta^m_n)] - \textrm{E}[H(\tilde{\Delta}^{0}_n, \, \ldots \, , \tilde{\Delta}^{m-1}_n, \hat{\Delta}^m_n)]| \le } \\
& & \qquad  = |\textrm{E}[ ( H(\Delta^{0}_n, \, \ldots \, , \Delta^m_n) - H(\tilde{\Delta}^{0}_n, \, \ldots \, , \tilde{\Delta}^{m-1}_n, \hat{\Delta}^m_n) ) \mathbb{I}_{\mathcal{B}^c_{n,s}}]| \\
& & \qquad \le C \|H\|_\oo \, n^{-\frac{1}{4}} \, . 
\end{eqnarray*}
Hence
$$
\lim_{n\ra \infty} |\textrm{E}[H(\Delta^{0}_n, \, \ldots \, , \Delta^m_n)] - \textrm{E}[H(\tilde{\Delta}^{0}_n, \, \ldots \, , \tilde{\Delta}^{m-1}_n, \hat{\Delta}^m_n)]| = 0 \, .
$$
It remains to prove that the second term in (\ref{I1}) converges to zero as $n$ goes to $+\oo$. Note that
\begin{eqnarray*}
\lefteqn{|\textrm{E}[H(\tilde{\Delta}^{0}_n, \, \ldots \, , \tilde{\Delta}^{m-1}_n, \hat{\Delta}^m_n)] - \textrm{E}[H(\tilde{\Delta}^{0}_n, \, \ldots \, , \tilde{\Delta}^{m-1}_n, \bar{\Delta}^m_n)]| = } \\
& & \qquad = |\textrm{E}[H(\tilde{\Delta}^{0}_n, \, \ldots \, , \tilde{\Delta}^{m-1}_n, \hat{\Delta}^m_n) - H(\tilde{\Delta}^{0}_n, \, \ldots \, , \tilde{\Delta}^{m-1}_n, \bar{\Delta}^m_n)]| \\
& & \qquad = |\textrm{E}[ ( H(\tilde{\Delta}^{0}_n, \, \ldots \, , \tilde{\Delta}^{m-1}_n, \hat{\Delta}^m_n) - H(\tilde{\Delta}^{0}_n, \, \ldots \, , \tilde{\Delta}^{m-1}_n, \bar{\Delta}^m_n) ) \mathbb{I}_{\mathcal{A}^c_{n,s}} ]| 
\end{eqnarray*}
The rightmost expression in the previous equality is bounded above by $C \|H\|_\oo \, n^{-\frac{1}{4}}$ plus
$$
|\textrm{E}[ ( H(\tilde{\Delta}^{0}_n, \, \ldots \, , \tilde{\Delta}^{m-1}_n, \hat{\Delta}^m_n) - H(\tilde{\Delta}^{0}_n, \, \ldots \, , \tilde{\Delta}^{m-1}_n, \bar{\Delta}^m_n) ) \mathbb{I}_{\mathcal{A}^c_{n,s} \cap \mathcal{B}_{n,s}} ]|
$$
which is equal to
\begin{equation}
\label{I4}
|\textrm{E}[ ( H(\Delta^{0}_n, \, \ldots \, , \Delta^m_n) - H(\Delta^{0}_n, \, \ldots \, , \Delta^{m-1}_n, \bar{\Delta}^m_n) ) \mathbb{I}_{\mathcal{A}^c_{n,s} \cap \mathcal{B}_{n,s}} ]| .
\end{equation}
To deal with the last expectation, define the coalescing times
$$
\tau_k = \inf \{ j \ge 1 : Z^{kn}_j = Z^{mn}_j \} \quad \textrm{and} \quad 
\bar{\tau}_k = \inf \{ j \ge 1 : Z^{kn}_j = \hat{Z}^{mn}_j \} \, ,
$$
for every $k \in {0,...,m-1}$.
The times $\tau_k$ and $\bar{\tau}_k$ have the tail of their distributions of $O(t^{-\frac{1}{2}})$. Also define
$$
\nu_k = \inf \big\{ 1\le j \le s \, n^2 : \, | \tilde{Z}^{l n}_j - \tilde{Z}^{k n}_j| \le n^{\frac{3}{4}} \big\}.
$$
Note that on $\mathcal{A}^c_{n,s}$ we have $\nu_k = \nu \le s \, n^2$, for some $k \in \{ 0, ... , m -1 \}$. Furthermore, on $\mathcal{B}_{n,s}$ and up to time $\nu$ , we have $(\Delta^{0}_n, \, \ldots \, , \Delta^m_n)$ equal to $(\Delta^{0}_n, \, \ldots \, , \Delta^{m-1}_n, \bar{\Delta}^m_n)$. We have that
$$
\textrm{P} \Big( \Big\{ \sup_{0\le t \le s} |\Delta^m_n(t) - \bar{\Delta}^m_n(t)| \ge \log(n) n^{-\frac{1}{8}} \Big\} \cap \mathcal{A}^c_{n,s} \cap \mathcal{B}_{n,s} \Big) 
$$
is equal to
\begin{eqnarray*}
\lefteqn{\textrm{P} \Big( \Big\{ \sup_{0\le j \le s \, n^2} |Z^{m n}_j - \bar{Z}^{m n}_j| \ge \log(n) n^{\frac{7}{8}} \Big\} \cap \mathcal{A}^c_{n,s} \cap \mathcal{B}_{n,s} \Big) = } \\
& & = \textrm{P} \Big( \Big\{ \sup_{\nu \le j \le s \, n^2} |Z^{m n}_j - \bar{Z}^{m n}_j| \ge \log(n) n^{\frac{7}{8}} \Big\} \cap \mathcal{A}^c_{n,s} \cap \mathcal{B}_{n,s} \Big) \\
& & \le \sum_{k=0}^{m-1} \textrm{P} \Big( \Big\{ \sup_{\nu_k \le j \le s \, n^2} |Z^{m n}_j - \bar{Z}^{m n}_j| \ge \log(n) n^{\frac{7}{8}} \Big\} \cap \mathcal{A}^c_{n,s} \cap \mathcal{B}_{n,s} \cap \{ \nu = \nu_k \} \Big) \, .
\end{eqnarray*}
For each $k \in \{0,..., m-1\}$, the respective term in the previous sum is bounded above by
\begin{eqnarray}
\label{I3}
& & \textrm{P} \Big( \Big\{ \sup_{\nu \le j \le s \, n^2} |Z^{mn}_j - \bar{Z}^{mn}_j| \ge \log(n) n^{\frac{7}{8}} \Big\} \cap \mathcal{A}^c_{n,s} \cap \mathcal{B}_{n,s} \cap \{\tau, \, \bar{\tau} \in [\nu , \nu + n^{\frac{7}{4}}] \} \Big) \nn \\
& & \textrm{P} \Big( \{ \tau_k > \nu_k + n^{\frac{7}{4}} \} \cup \{ \bar{\tau}_k > \nu_k + n^{\frac{7}{4}} \} \Big) .
\end{eqnarray}
The second term in (\ref{I3}) is bounded above by $2 \,\frac{n^{\frac{3}{4}}}{n^{\frac{7}{8}}} = 2 \, n^{-\frac{1}{8}}$ and the first by
\begin{eqnarray*}
& & \textrm{P} \Big( \Big\{ \sup_{\nu \le j \le (\nu + n^{\frac{7}{4}}) \wedge s \, n^2 } |Z^{mn}_j - \bar{Z}^{mn}_j| \ge \log(n) n^{\frac{7}{8}} \Big\} \cap \mathcal{A}^c_{n,s} \cap \mathcal{B}_{n,s} \Big) \\
& & \quad \le 2 \textrm{P} \Big( \sup_{0 \le j \le (\nu + n^{\frac{7}{4}})} \frac{|Z^0_j|}{n^{\frac{7}{8}}} \ge \frac{\log(n)}{2} \Big) \, ,
\end{eqnarray*}
which by Donsker invariance principle goes to zero as $n$ goes to infinity. Finally we have that (\ref{I4}) is bounded above by a term that converges to zero as $n \ra \infty$ plus
$$
\Big| \textrm{E} \Big[ \big( H(\Delta^{0}_n, \, \ldots \, , \Delta^m_n) - H(\Delta^{0}_n, \, \ldots \, , \Delta^{m-1}_n, \bar{\Delta}^m_n) \big) \mathbb{I}_{ \mathcal{E}_{n,s} } \Big] \Big| \, ,
$$
where 
$$
\mathcal{E}_{n,s} =
\Big\{ \sup_{0\le t \le s} |\Delta^m_n(t) - \bar{\Delta}^m_n(t)| \le \log(n) n^{-\frac{1}{8}} \Big\} \cap \mathcal{A}^c_{n,s} \cap \mathcal{B}_{n,s}
$$
By the uniform continuity of $H$ the rightmost expectation in the previous expression converges to zero as $n$ goes to $+\infty$.


\bigskip 

\begin{thebibliography}{99}

\bibitem{a}
R.~Arratia: Coalescing Brownian motions and the voter model
on $\Z$. Unpublished partial manuscript (circa 1981), available
at {\tt rarratia@math.usc.edu}.

\bibitem{a1}
R.~Arratia: Limiting point processes for rescalings of
coalescing and annihilating random walks on $\Z^d$, {\it Ann.~Prob.}~{\bf 9},
pp 909-936 (1981).

\bibitem{bmsv}
S.~Belhaouari; T.~Mountford; R.~Sun; G.~Valle:
Convergence results and sharp estimates for the voter model interfaces,
{\it Electron.~J.~Probab.}~{\bf  11}, 768--801 (2006)

\bibitem{cfd} C. Coletti; L. R. Fontes; E. Dias: Scaling limit for a drainage network model, {\it J. Appl.Prob.}~{\bf  46}, 1184--1197 (2009).

\bibitem{Du}
R.~Durrett: {\it Probability: Theory and Examples}. Second
Edition, Duxbury Press, (1996).

\bibitem{ffw} 
P.~A.~Ferrari; L.~R.~G.~Fontes; X.~Y.~Wu: Two dimensional Poisson trees converge to the Brownian web,
{\it Ann.~Inst. H. Poincar\'e Probab. Statist.}~{\bf 41}, 851-858 (2005).

\bibitem{finr}
L.~R.~G.~Fontes; M.~Isopi; C.~M.~Newman; K.~Ravishankar: The
Brownian web, {\it Proc. Natl. Acad. Sci. USA} {\bf 99}, no.~25,
15888 -- 15893 (2002).

\bibitem{finr1}
L.~R.~G.~Fontes; M.~Isopi; C.~M.~Newman; K.~Ravishankar: The
Brownian web: characterization and convergence, {\it Ann.~Probab.}~{\bf 32},
2857-2883 (2004).

\bibitem{finr2}
L.~R.~G.~Fontes; M.~Isopi; C.~M.~Newman; K.~Ravishankar:
Coarsening, nucleation, and the marked Brownian web. 
{\it Ann.~Inst. H. Poincar\'e Probab. Statist.}~{\bf 42}, 37-60 (2006).

\bibitem{grs}
S.~Gangopadhyay; R.~Roy; A.~Sarkar: Random oriented trees: a model
of drainage networks, {\it Ann.~App.~Probab.}~{\bf 14}, 1242-1266 (2004).

\bibitem{nrs}
C.~M.~Newman; K.~Ravishankar; R.~Sun: Convergence of coalescing nonsimple
random walks to the Brownian web,
{\it Electron.~J.~Probab.} {\bf 10}, 21--60 (2005)

\bibitem{s} R.~Sun: Convergence of Coalescing Nonsimple Random Walks to the Brownian Web, Ph.D. Thesis, Courant Institute of Mathematical Sciences, New York University, {\it arXiv:math/0501141}, (2005).

\bibitem{ss} A.~Sarkar, R.~Sun: Brownian web in the scaling limit of supercritical oriented percolation in dimension 1+1, preprint.

\bibitem{ss1} R.~Sun, J.~M.~Swart: The Brownian net, {\it Ann.~Probab.}~{\bf 36},
1153-1208 (2008).

\bibitem{tw}
B.~T\'oth; W.~Werner: The true self-repelling motion,
{\it Prob.~Theory Related Fields}~{\bf 111}, pp 375-452 (1998).

\end{thebibliography}
\end{document}